\RequirePackage{fix-cm}

\documentclass[]{scrartcl}

\usepackage[utf8]{inputenc}
\usepackage{amsmath}
\usepackage{amsthm}
\usepackage{amssymb}
\usepackage{tikz}
\usepackage{subfigure}
\usepackage{proof}

\newtheorem{definition}{Definition}[section]
\newtheorem{theorem}[definition]{Theorem}
\newtheorem{lemma}[definition]{Lemma}
\newtheorem{corollary}[definition]{Corollary}
\newtheorem{remark}[definition]{Remark}

\newcommand{\id}{\mathop{\mathrm{D}_i}}

\newcommand{\ind}{\mathop{\mathrm{I}}}
\newcommand{\iso}{\mathop{\mathrm{iso}}}

\newcommand{\essi}{\mathop{\mathrm{Ess}_i}}

\newcommand{\join}{\mathop{\mathrm{*}}}
\newcommand{\corona}{\mathop{\mathrm{\circ}}}

\begin{document}
\title{The Independent Domination Polynomial}

\author{Markus Dod \thanks{This research was supported by the ESF and a scholarship from the State of Saxony.}\\University of Applied Sciences Mittweida\\ Faculty Mathematics/Sciences/Computer Science\\mdod@hs-mittweida.de}

\date{29.10.2015}

\maketitle

\begin{abstract}
A vertex subset $W\subseteq V$ of the graph $G=(V,E)$ is an independent 
dominating set if every vertex in $V\backslash W$ is adjacent
to at least one vertex in $W$ and the vertices of $W$ are pairwise non-adjacent. 
The independent domination polynomial is the 
ordinary generating function for the number of independent dominating sets 
in the graph. We investigate in this paper properties of the independent domination 
polynomial and some interesting connections to well known counting problems.
\end{abstract}

\section{Introduction}
We consider finite simple undirected graphs and identify edges with two-element subsets of the vertex set. We call a graph \emph{non-trivial} if it has a least one edge. Let $G=(V,E)$ be a graph with vertex set $V$ and edge set $E$. Let $W\subseteq V$, then the \emph{(vertex) induced subgraph} $G[W]$ is the graph
\[
G[W] = (W, \{\{u,v\}\in E: u,v\in W\}).
\]
We denote the number of isolated vertices in $G$ by $\mathrm{iso}(G)$. The \emph{open neighborhood} $N_{G}(v)$ of a vertex 
$v\in V$ is the set of all vertices that are adjacent to $v$ in $G$. Analogously, we define
\[
	N_{G}(W)=\bigcup_{v\in W}N_{G}(v)
\]
for any vertex subset $W\subseteq V$. The \emph{closed neighborhood} $N_{G}[W]$ of a vertex subset $W\subseteq V$ is simply the set $N_{G}(W)\cup W$. If the graph is known from the context, we write $N(v)$ and $N(W)$ instead of $N_{G}(v)$ and $N_{G}(W)$, respectively. The maximum degree $\Delta(G)$ of the graph $G$ is defined as $\max_{v\in V}{|N(v)|}$.

A vertex set $W\subseteq V$ is called an \emph{independent dominating set} of $G$ if 
$N_{G}[W]=V$ and the induced subgraph has only isolated vertices, in other words $\mathrm{iso}(G[W])=|W|$. The \emph{independent domination polynomial} of a graph $G$ is the 
ordinary generating function for the number of independent dominating sets of $G$:
\[
	\id(G,x)=\sum_{\substack{W\subseteq V \\N_{G}[W]=V\\\mathrm{iso}(G[W])=|W|}}x^{|W|}.
\]

The independent domination number of the graph is a well known graph parameter (see \cite{Goddard2013} for an overview). But little is known about counting independent dominating sets. The independent domination polynomial can be obtained from the trivariate total domination polynomial (see \cite{Dod2014a}).

The paper is organized as follows. In the following section we introduce some basic properties of the independent domination polynomial. In Section \ref{section::recurrenceEquations} we prove some basic recurrence equations and in Section \ref{section::id-specialgraphclasses} we give equations for some special graph classes.

\section{The independent domination polynomial}
Like other graph polynomials the independent domination polynomial is multiplicative in respect to the components of the graph (which also follows from the connection to the trivariate total domination polynomial \cite{Dod2014a}).

\begin{lemma}\label{lemma::id-twocomponents}
Let $G=(V,E)$ be a graph with two components $G_1=(V_1,E_1)$ and $G_2=(V_2,E_2)$. Then
\begin{equation*}
\id(G,x) = \id(G_1,x) \id(G_2,x).
\end{equation*}
\end{lemma}
\begin{proof}
The proof of the lemma follows directly from the definition of the polynomial.
\end{proof}

The join of two graphs was first defined by Harary in 1969.
\begin{definition}\cite{Harary1969}
The join $G \join H$ of two graphs $G=(V(G),E(G))$ and $H=(V(H),E(H))$ is the graph union $G\cup H$ together with all the edges joining $V(G)$ and $V(H)$.
\end{definition}

\begin{theorem}
Let $G=(V(G),E(G))$ and $H=(V(H),E(H))$ be two non-trivial graphs. Then
\begin{equation*}
\id(G \join H,x) = \id(G,x)+\id(H,x).
\end{equation*}
\end{theorem}
\begin{proof}
It suffices to observe that any independent dominating set of $G \join H$ is completely contained in exactly one of the vertex sets $V(G)$ or $V(H)$.
\end{proof}

Frucht and Harary introduced 1970 the corona of two graphs.
\begin{definition}\cite{Frucht1970}
Let $G=(V(G),E(G))$ and $H=(V(H),E(H))$ be graphs. Then the corona of $G$ and $H$ is the graph $G\corona H$ which is the disjoint union of $G$ and $|V(G)|$ copies of $H$ and every vertex $v$ of $G$ is adjacent to every vertex in the corresponding copy of $H$.
\end{definition}

The independence polynomial is the ordinary generating function for the number of independent sets in the graph and is denoted by $\ind(G,x)$ (see \cite{Levit2005}).

\begin{theorem}
Let $G=(V(G),E(G))$ and $H=(V(H),E(H))$ be graphs, $|V(G)|\geq 1$, $|V(H)|\geq 1$ and $n=|V(G)|$. Then
\begin{equation*}
\id(G \corona H,x) = \id(H,x)^{n} \ind\left(G,\frac{x}{\id(H,x)}\right).
\end{equation*}
\end{theorem}
\begin{proof}
Every independent vertex subset of $G$ can be expanded to an independent dominating set in $G \circ H$. Let $S\subseteq V(G)$ be such an independent set, then $S$ together with dominating vertices adjacent to the vertices in $V(G)-S$ form an independent dominating set. Precisely if $|S|=k$, then  in $(n-k)$ copies of $H$ an independent dominating set must exist. Let $i_k$ be the coefficient of $x^k$ in $\ind(G,x)$. Then 
\begin{align*}
\id(G \corona H,x) =& \sum \limits_{k=0}^{n} i_k x^k \id(H,x)^{n-k}\\
=& \id(H,x)^{n} \sum \limits_{k=0}^{n}i_k x^{k}\id(H,x)^{-k}\\
=& \id(H,x)^{n} \ind\left(G,\frac{x}{\id(H,x)}\right)
\end{align*}
and therefore the theorem follows.
\end{proof}

\begin{corollary}
Let $G=(V,E)$ be a graph with $n$ vertices and $E_r$ an edgeless graph with $r$ ($r>0$) vertices. Then
\begin{equation*}
\id(G \corona E_r,x) = x^{rn} \ind(G,x^{1-r}).
\end{equation*}
\end{corollary}
 
\begin{definition}\cite{Goddard2013}
Let $G=(V,E)$ be a graph. Then the r-expansion $\exp(G,r)$ is the graph obtained from $G$ by replacing every vertex $v\in V$ with an independent set $I_v$ of size $r$ and replacing every edge $uv\in E$ with a complete bipartite graph with the bipartite sets $I_u$ and $I_v$.
\end{definition}

\begin{theorem}
Let $G=(V,E)$ be a graph and $\exp(G,r)$ its r-expansion. Then
\begin{equation*}
\id(\exp(G,r),x)=\id(G,x^r)
\end{equation*}
\end{theorem}
\begin{proof}
Let $W\subseteq V$ be an independent dominating set in $G$. Then in $\exp(G,r)$ all vertices in $I_w$, for $w\in W$, must be dominating and all vertices in $I_u$, for $u\in V-W$, are non-dominating (because of the complete bipartite graphs between vertices in $I_w$ and $I_u$). Therefore, every independent dominating set in $G$ can be expanded to exactly one independent dominating set in $\exp(G,r)$ and vice versa.
\end{proof}

Kotek et al. proved in \cite{Kotek2013a} that the alternating sum over the domination polynomials of the vertex induced subgraphs equals $1+(-x)^n$. In contrast to this result the alternating sum over the independent domination polynomials equals one.
\begin{theorem}\label{theorem::id-suminducedsubgraphs}
Let $G=(V,E)$ be a connected graph with at least two vertices. Then
\begin{equation*}
\sum \limits_{W\subseteq V} (-1)^{|W|} \id(G[W],x) = 1.
\end{equation*}
\end{theorem}
\begin{proof}
First of all we insert the definition of the independent domination polynomial in the equation and change the order of the summation.
\allowdisplaybreaks{
\begin{align}
\sum \limits_{W\subseteq V} (-1)^{|W|} \id(G[W],x) =& \sum \limits_{W\subseteq V} (-1)^{|W|} \sum \limits_{\substack{U\subseteq W\\N_{G[W]}[U]=W\\\iso(G[U])=|U|}} x^{|U|}\notag\\
=& \sum \limits_{\substack{U\subseteq V\\\iso(G[U])=|U|}} x^{|U|} \sum \limits_{\substack{W: U\subseteq W\\N_{G[W]}[U]=W}} (-1)^{|W|}\notag\\
=& \sum \limits_{\substack{U\subseteq V\\\iso(G[U])=|U|}} x^{|U|} \sum \limits_{W: U\subseteq W \subseteq N_{G[W]}[U]} (-1)^{|W|}\label{eq::int1}\\
=& \sum \limits_{\substack{U\subseteq V\\\iso(G[U])=|U|}} x^{|U|} \sum \limits_{W: U\subseteq W \subseteq N_{G}[U]} (-1)^{|W|}\label{eq::int2}\\
=& \sum \limits_{\substack{U\subseteq V\\\iso(G[U])=|U|}} (-x)^{|U|} \sum \limits_{Y \subseteq N_{G}(U)} (-1)^{|Y|}\notag\\
=& 1\notag.
\end{align}}
In Equation (\ref{eq::int1}) we sum over all vertex subsets $W$ such that $U$ is an independent dominating set in $G[W]$. The condition $W \subseteq N_{G[W]}[U]$ in Equation (\ref{eq::int1}) guarantees that we sum only over subsets $W$ such that $U$ is a dominating set in $G[W]$. Hence, in the inner sum $W$ can be every subset from $N_G[U]$. With these considerations we obtain Equation (\ref{eq::int2}). Because of the fact that $U$ is included in every subset $W$ of the inner sum, the summation is performed only over vertex subsets included in $N_G(U)$ and $(-1)^{|U|}$ is factored out from the inner sum. The second sum vanishes for every set $U$ which is not equal $V$ or $\emptyset$ and therefore we obtain the theorem.
\end{proof}

\begin{remark}\label{remark::id-suminducedsubgraphs}
Let $G=(\{v\},\emptyset)$ be a graph with one vertex. Then
\begin{equation*}
\sum \limits_{W\subseteq V} (-1)^{|W|} \id(G[W],x) = 1-x.
\end{equation*}
\end{remark}

If the graph has more than one component we get as a consequence of Theorem \ref{theorem::id-suminducedsubgraphs} and Remark \ref{remark::id-suminducedsubgraphs} the following corollary.
\begin{corollary}\label{corollary::id-suminducedsubgraphs}
Let $G=(V,E)$ be a graph. Then
\begin{equation*}
\sum \limits_{W\subseteq V} (-1)^{|W|} \id(G[W],x) = (1-x)^{\iso(G)}.
\end{equation*}
\end{corollary}

Applying M\"obius inversion to Corollary \ref{corollary::id-suminducedsubgraphs} gives the next corollary.
\begin{corollary}\label{corollary::id-alternatingsum}
Let $G=(V,E)$ be a graph. Then
\begin{equation*}
\id(G,x) = \sum \limits_{W\subseteq V} (-1)^{|W|} (1-x)^{\iso(G[W])}.
\end{equation*}
\end{corollary}

The previous corollary gives us a formula to calculate the coefficients of the independent domination polynomial.
\begin{corollary}
Let $G=(V,E)$ be a graph with $n$ vertices. Then
\begin{equation*}
\id(G,x) = \sum \limits_{k=0}^{n} x^k \sum \limits_{\substack{W\subseteq V\\\iso(G[W])\geq k}} (-1)^{|W|+k} \binom{\iso(G[W])}{k}.
\end{equation*}
\end{corollary}
\begin{proof}
Using the Corollary \ref{corollary::id-alternatingsum}, we get
\begin{align*}
\id(G,x) &= \sum \limits_{W\subseteq V} (-1)^{|W|} (1-x)^{\iso(G[W])}\\
&= \sum \limits_{W\subseteq V} (-1)^{|W|} \sum \limits_{k=0}^{\iso(G[W])} \binom{\iso(G[W])}{k} (-x)^k\\
&= \sum \limits_{k=0}^{n} (-x)^k \sum \limits_{W\subseteq V} (-1)^{|W|} \binom{\iso(G[W])}{k}\\
&= \sum \limits_{k=0}^{n} x^k \sum \limits_{\substack{W\subseteq V\\\iso(G[W])\geq k}} (-1)^{|W|+k} \binom{\iso(G[W])}{k}.
\end{align*}
\end{proof}

We can use these results to prove a theorem which offers a fast way to calculate the independent domination polynomial. This theorem uses the \emph{i}-essential sets of a graph. The concept of essential sets of a graph was introduced by Kotek, Preen, and Tittmann \cite{Kotek2013a} for the calculation of the domination polynomial.
\begin{definition}
Let $G=(V,E)$ be a graph and $W$ a vertex subset of the graph. The set $W$ is called \emph{i}-essential if $W$ contains the open neighborhood of at least one vertex of $V\backslash W$. We denote by $\essi(G)$ the family of \emph{i}-essential sets of $G$, in formula:
\begin{equation*}
\essi(G) = \{X\subseteq V: \exists v\in V\backslash X: N(v) \subseteq X\}.
\end{equation*}
\end{definition}

An open problem concerning \emph{i}-essential sets is: Can the number of \emph{i}-essential sets of a given graph be calculated, without calculating the sets themselves? The next lemma gives two basic properties of the \emph{i}-essential sets, which may be helpful to get a better understanding of these sets.
\begin{lemma}
Let $G=(V,E)$ be a graph. Then
\begin{equation*}
\min_{W\in \essi(G)} \{|W|\} = \delta(G)
\end{equation*}
and
\begin{equation*}
N(v) \in \essi(G) \quad \forall v \in V.
\end{equation*}
\end{lemma}

\begin{remark}
If $W$ is an independent dominating set of the graph $G$ and $W\subset U$, then $U$ is not an independent dominating set. Let $S$ be the partial ordered set $(\mathcal{P}(V), \subseteq)$, then the set of the independent dominating sets of $G$ is an anti-chain in $S$.
\end{remark}

\begin{theorem}
Let $G=(V,E)$ be a graph with $n$ vertices, then
\begin{equation*}
\id(G,x) = (-1)^n \sum \limits_{U\subseteq \essi(G)} (-1)^{|U|} \left((1-x)^{|\{v\in V\backslash U|N_G(v)\subseteq U\}|}-1\right).
\end{equation*}
\end{theorem}
\begin{proof}
We obtain by Corollary \ref{corollary::id-alternatingsum}:
\begin{align*}
\id(G,x) =& \sum \limits_{W\subseteq V} (-1)^{|W|} (1-x)^{\iso(G[W])}\\
=& \sum \limits_{U\subseteq V} (-1)^{|V\backslash U|} (1-x)^{\iso(G[V\backslash U])}\\
=& (-1)^n \sum \limits_{U\subseteq V} (-1)^{|U|} (1-x)^{|\{v\in V\backslash U|N_G(v)\subseteq U\}|}\\
=& (-1)^n \sum \limits_{U\subseteq \essi(G)} (-1)^{|U|} \left((1-x)^{|\{v\in V\backslash U|N_G(v)\subseteq U\}|}-1\right).
\end{align*}

The second factor in the sum vanishes if and only if $\{v\in V\backslash U:N_G(v)\subseteq U\}=\emptyset$. Consequently, only \emph{i}-essential sets contribute to the sum.
\end{proof}

\section{Recurrence equations}\label{section::recurrenceEquations}
In this section we prove several recurrence equations for the independent domination polynomial. We need the following seven graph operations:
\begin{itemize}
\item $G-v$ denotes the graph obtained from $G$ by removal of the vertex $v\in V$ and all edges incident with $v$.
\item $G/v$ denotes the graph obtained from $G$ by the removal of the vertex $v\in V$ and the addition of edges between any pair of non-adjacent neighbors of $v$.
\item $G\odot v$ denotes the graph obtained from $G$ by removing all edges between adjacent vertices of $v\in V$.
\item $G\circ v$ denotes the graph obtained from $G$ by removing $v$ and the addition of loops to all neighbors of $v$.
\item $G-N[v]$ denotes the graph $G-N_G[v]$ obtained by deleting all of the vertices in the closed neighborhood of the vertex $v$ and the edges incident to them.
\item $G-e$ denotes the graph obtained from $G$ by removing the edge $e\in E$.
\item $G\circ v$ denotes the graph obtained from $G$ by removing $v$ and adding a loop to every neighbor of $v$.
\end{itemize}

A loop in the context of domination means that the vertex dominates itself. If a vertex $v$ has a loop, then $v\in N(v)$ and hence $v$ can not be in any independent dominating set. Hence, $\id(G\circ v,x)$ is the independent domination polynomial of the graph $G-v$ under the condition that no vertex in $N(v)$ is dominating.

\begin{remark}
Let $G=(V,E)$ be a graph and $v\in V$. Then
\begin{equation*}
\id((G\odot v)\circ v,x) = \id(G\circ v,x).
\end{equation*}
\end{remark}

\begin{remark}\cite{Kotek2012}
Let $G=(V,E)$ be a graph and $v\in V$. Then
\begin{equation}\label{eqn::basics-operations-facts}
(G\odot v) - N[v] \cong G-N[v].
\end{equation}
\end{remark}

\begin{theorem}\label{theorem::id-recurrence}
Let $G=(V,E)$ be a graph and $v$ a vertex of the graph. Then
\begin{equation*}
\id(G,x) = \id(G-v,x) - \id(G\circ v,x) + x \id(G-N[v],x).
\end{equation*}
\end{theorem}
\begin{proof}
If the vertex $v$ is dominating, then it dominates all vertices in the neighborhood and these vertices can not be dominating. This case will be counted by $x \id(G-N[v],x)$. If the vertex $v$ is not dominating, then at least one vertex in $N(v)$ must be dominating. The polynomial $\id(G-v,x)$ counts these independent dominating sets, but it counts also those sets where in $N(v)$ no vertex is dominating. Hence, we must subtract the polynomial for these cases to get the theorem.
\end{proof}

The next corollary follows directly from the last theorem and gives a recurrence equation under the condition that the neighborhood of the vertex $v$ has some special properties.
\begin{corollary}
Let $G=(V,E)$ be a graph, $u,v\in V$, $u\neq v$ and $N(u) = N(v)$. Then
\begin{equation*}
\id(G,x) = \id(G-v,x) + (x^2 -x) \id(G-N[v]-u,x).
\end{equation*}
\end{corollary}
\begin{proof}
Let $u$ and $v$ two vertices of the graph with $N(u) = N(v)$, then in the graph $G-N[v]$ the vertex $u$ is isolated. Hence, $u$ is included in all independent dominating sets of $G-N[v]$. The polynomial $x\id(G-N[v]-u,x)$ counts this case and therefore $x\id(G-N[v]-u,x) = \id(G\circ v,x)$.
\end{proof}

If we use the $\odot$-operation we can prove the following theorem.
\begin{theorem}\label{theorem::id-recurrence2}
Let $G=(V,E)$ be a graph and $v\in V$. Then
\begin{equation*}
\id(G,x) = \id(G-v,x) + \id(G\odot v,x) - \id(G\odot v -v,x).
\end{equation*}
\end{theorem}
\begin{proof}
Applying the Equations (\ref{eqn::basics-operations-facts}) to Theorem \ref{theorem::id-recurrence} gives
\begin{align}
\id(G,x) - \id(G-v,x) =& x \id(G-N[v],x) - \id(G\circ v,x) \notag\\
=& x \id(G-N[v],x)- \id((G\odot v) \circ v,x)\label{eqn::id-proof-rec1}.
\end{align}

Now we apply Theorem \ref{theorem::id-recurrence} to the graph $G\odot v$
\begin{align}
\id(G\odot v,x) - \id((G\odot v)-v,x) =& x \id((G\odot v)-N[v],x) - \id((G\odot v)\circ v,x)\label{eqn::id-proof-rec2}.
\end{align}
Putting the Equations (\ref{eqn::id-proof-rec1}) and (\ref{eqn::id-proof-rec2}) together gives the theorem.
\end{proof}

It is also possible to prove a theorem which gives a recurrence equation for the deletion of an edge in the graph.
\begin{theorem}
Let $G=(V,E)$ be a graph and $e=\{u,v\}\in E$. Then
\begin{align*}
\id(G,x) =& \id(G-e,x) - x^2 \id(G-N[u,v],x)\\
&+ x \id(G\circ v - N[u],x) + x \id(G\circ u - N[v],x).
\end{align*}
\end{theorem}
\begin{proof}
Every independent dominating set from $G$ will be an independent dominating set in $G-e$, except of those sets where $u$ or $v$ are dominating. $u$ and $v$ can be dominating in $G-e$, but not in $G$. Therefore, we must subtract $x^2 \id(G-N[u,v],x)$. Suppose now that only one of these two vertices are dominating and no vertex in the neighborhood of the other vertex is dominating. This situation will be counted in the graph $G$ but not in the graph $G-e$. Hence, we must add the polynomial for this case and the theorem follows. Remark that $x \id(G\circ v - N[u],x)$ is the independent domination polynomial under the condition that the vertex $u$ is dominating and no vertex in the neighborhood of $v$ (except $u$) is dominating.
\end{proof}

\begin{corollary}
Let $G=(V,E)$ be a graph, $e=\{u,v\}$ be an edge of the graph and $N[u]=N[v]$. Then
\begin{equation*}
\id(G,x) = \id(G-e,x) + (2x-x^2)\id(G-N[u],x).
\end{equation*}
\end{corollary}

\section{Special graph classes}\label{section::id-specialgraphclasses}
In general the calculation of the independent domination polynomial is in $\#P$ \cite{Goddard2013}, but for some special graph classes we can prove some nice recursive or closed equations. For the edgeless graph $E_n$ the independent domination polynomial is simply $x^n$. In the complete graph every independent dominating set has the size one and therefore
\begin{equation}\label{eqn::id-completeGraph}
\id(K_n,x) = nx.
\end{equation}

\begin{theorem}\label{theorem::id-completebipartite}
Let $K_{pq}=(V_1\cup V_2,E)$ be the complete bipartite graph and $p,q\geq 1$. Then
\begin{equation*}
\id(K_{pq},x) = x^p+x^q.
\end{equation*}
\end{theorem}
\begin{proof}
If in $V_1$ at least one vertex is dominating, then in $V_2$ all vertices are dominated. Therefore, all vertices in $V_1$ must be dominating so that they are a dominating set in the graph. The same argumentation holds if at least one vertex in $V_2$ is dominating.
\end{proof}

\begin{theorem}
Let $G=(V,E)$ be the path $P_n$ with at least four vertices. Then
\begin{equation*}
\id(P_n,x) = x\id(P_{n-2},x)+x\id(P_{n-3},x),
\end{equation*}
with the initial conditions
\begin{equation*}
\id(P_1,x) = x,\textnormal{ } \id(P_2,x) = 2x \textnormal{ and } \id(P_3,x) = x^2+x.
\end{equation*}
\end{theorem}
\begin{proof}
If the first vertex of the path is dominating, then the second is dominated and therefore it can not be dominating. This case will be counted by $x\id(P_{n-2},x)$. If the first vertex is non-dominating, then the second vertex must be dominating. This gives $x\id(P_{n-3},x)$ and the theorem follows.
\end{proof}

Moreover, we can use prove an explicit formula for the independent domination polynomial of the path $P_n$.
\begin{theorem}\label{theorem::path}
Let $G=(V,E)$ be the path $P_n$ with $n\geq 2$. Then
\begin{equation}\label{eqn::id-path}
\id(P_n,x)=\sum_{k=1}^{\lfloor (n+3)/2\rfloor} \binom{k+1}{n-2k+1}x^{k}.
\end{equation}
\end{theorem}
\begin{proof}
Let $p(n, k)$ be the number of independent dominating sets $W$ of $P_n$ with exactly $k$ vertices. From the defnition it follows that between the vertices of $W$ there has to be a vertex not in $W$. Additionally, there have to be $n - k - (k - 1)$ other vertices not in $W$, one or none of them may be "before" the rest vertex in $W$, between such vertices (so that there are altogether two of them) or "behind" the last vertex in $W$. Hence there are $k + 1$ possible positions. It follows that $p(n,k) = \binom{k+1}{n-2k+1}$. Consequently,
\begin{equation*}
\id(P_n,x)=\sum_{k=1}^{\lfloor (n+3)/2\rfloor} p(n,k) x^k =\sum_{k=1}^{\lfloor (n+3)/2\rfloor} \binom{k+1}{n-2k+1} x^k.
\end{equation*}
\end{proof}

We can use the polynomial of the path $P_n$ to prove a theorem for the cycle $C_n$.
\begin{theorem}
Let $G=(V,E)$ be the cycle $C_n$ ($n\geq 7$). Then
\begin{equation*}
\id(C_n,x) = 2x\id(P_{n-3},x)+x^2\id(P_{n-6},x).
\end{equation*}
\end{theorem}
\begin{proof}
Suppose we have a numbering of the vertices. Starting with one up to $n$. If the vertex $1$ of the cycle is dominating, then its two neighbors $2$ and $n$ are dominated and they cannot be dominating. This case will be counted by $x\id(P_{n-3},x)$. If the vertex $1$ is non-dominating, then one of its neighbors must be dominating. If the vertex $2$ is dominating, then the vertex $n$ can be dominating or not. This case will be counted by $x\id(P_{n-3},x)$. If the vertices $1$ and $2$ are non-dominating, then the vertices $3$ and $n$ must be dominating. This gives the last part of the sum and the theorem is proved.
\end{proof}

Using Equation (\ref{eqn::id-path}) gives
\begin{equation*}
\id(C_n,x) = \sum_{k=0}^{\lfloor \frac{n-2}{2}\rfloor }\left( 2 \binom{k+2}{n-2k-4}+\binom{k+1}{n-2k-5}\right) x^{k+2}.
\end{equation*}

\section{Conclusion and open problems}
For the independent domination polynomial of some graph products, nice theorems are known. Is it possible to find more such results for other products or for the cartesian product $G\Box H$?

We introduced the $\circ$-operation for vertices of the graph and get a recurrence equation in respect to this operation. Is it possible to prove similar recurrence equations for the domination polynomial or the total domination polynomial?

\section*{Acknowledgement}
The author would like to thank Peter Tittmann for very helpful ideas and discussions which improved the paper.
Moreover, the author would like to express his gratitude to Manja Reinwardt and Ester Then for their careful reading and helpful comments.

\nocite{*}

\end{document}